\documentclass[a4paper,11pt]{article}
\usepackage{amsmath,amssymb}
\usepackage{amsfonts}
\usepackage{eucal}
\usepackage{amsthm}
\numberwithin{equation}{section}

\newcommand{\bigpare}[1]{\bigl(#1\bigr)}
\newcommand{\biggpare}[1]{\biggl(#1\biggr)}

\newcommand{\Bigbra}[1]{\Bigl\{#1\Bigr\}}

\newcommand{\bigbrac}[1]{\bigl[#1\bigr]}

\newcommand{\bigset}[2]{\bigl\{#1\bigm|#2\bigr\}}

\newcommand{\Bignorm}[1]{\Bigl\| #1 \Bigr\|}

\newcommand{\bigabs}[1]{\bigl| #1 \bigr|}

\newcommand{\jap}[1]{\langle #1 \rangle}

\def\a{\alpha}
\def\b{\beta}

\def\d{\delta}
\def\e{\varepsilon}
\def\f{\varphi}
\def\g{\psi}

\def\i{\mbox{\raisebox{.5ex}{$\chi$}}}

\def\l{\lambda}
\def\m{\mu}

\def\s{\sigma}

\def\x{\xi}
\def\y{\eta}

\renewcommand{\L}{\Lambda}
\renewcommand{\O}{\Omega}
\renewcommand{\S}{\Sigma}

\newcommand{\OPS}{O\!P\!S}
\newcommand{\Op}{\mathrm{Op}}

\newcommand{\esssupp}{\mathrm{ess}\text{-}\mathrm{supp}}
\newcommand{\torus}{\mathbb{T}}

\def\re{\mathbb{R}}
\def\co{\mathbb{C}}
\def\ze{\mathbb{Z}}
\def\na{\mathbb{N}}
\def\pa{\partial}

\renewcommand{\Im}{\mathrm{Im}\,}

\newcommand{\supp}{\mathrm{{supp}}}

\newcommand{\dto}{\downarrow}

\DeclareMathOperator*{\slim}{s-lim}

\newcommand{\Ran}{\mathrm{Ran\;}}
\newcommand{\WF}{\mathrm{WF}}

\newtheorem{thm}{Theorem}[section]
\newtheorem{lem}[thm]{Lemma}
\newtheorem{prop}[thm]{Proposition}
\newtheorem{cor}[thm]{Corollary}

\theoremstyle{definition}

\newtheorem{ass}{Assumption}

\newtheorem{example}{Example}

\theoremstyle{remark}
\newtheorem{rem}{Remark}[section]

\numberwithin{equation}{section}

\title{Microlocal properties of scattering matrices\footnote{2010 Mathematics Subject Classification 
58J50, 35P25, 81U05.}}
\author{
Shu N{\sc akamura}%
\footnote{
Graduate School of Mathematical Science, 
University of Tokyo, Tokyo, Japan, 
Email: {\tt shu@ms.u-tokyo.ac.jp}. Partially supported by JSPS Grant Kiban (A) 21244008.}}

\begin{document}
\maketitle

\begin{abstract}
We consider scattering theory for a pair of operators $H_0$ and $H=H_0+V$ on $L^2(M,m)$, 
where $M$ is a Riemannian manifold, $H_0$ is a multiplication operator on $M$ and 
$V$ is a pseudodifferential operator of order $-\m$, $\m>1$. We show that a time-dependent 
scattering theory can be constructed, and the scattering matrix is a pseudodifferential 
operator on each energy surface. Moreover, the principal symbol of the scattering matrix is given 
by a Born approximation type function. The main motivation of the study comes from applications to 
discrete Schr\"odigner operators, but it also applies to various differential operators with 
constant coefficients and short-range perturbations on Euclidean spaces. 
\end{abstract}


\section{Introduction}

Let $M$ be a smooth $d$-dimensional complete Riemannian manifold with a smooth density $m$, and let
$p_0(\x)$, $\x\in M$, be a real-valued smooth function on $M$.  We define 
\[
H_0\f(\x)= p_0(\x)\f(\x), \quad \f\in D(H_0)=\bigset{\f\in L^2(M,m)}{p_0\f\in L^2}
\]
be the multiplication operator by $p_0$ on $\mathcal{H}=L^2(M,m)$. 
It is easy to see that $H_0$ is self-adjoint. Let $\tilde V=V(-D_\x,\x)$ be a pseudodifferential operator 
with a symbol $V\in S^{-\m}_{1,0}(M)$ with $\m>1$. We suppose $\tilde V$ is an $H_0$-bounded self-adjoint 
operator on $\mathcal{H}$, and hence the principal symbol of $\tilde V$ may be supposed to be real-valued. 
We write $\tilde V$ and $V$ by the same symbol for simplicity. 
We set
\[
H= H_0+V\quad \text{on } \mathcal{H},
\]
and $H$ is self-adjoint with $D(H)=D(H_0)$. We write 
\[
v(\x)= dp_0(\x)\in T^*M
\quad\text{and}\quad 
M_0=\bigset{\x\in M}{v(\x)\neq 0}. 
\]
Let $I$ be a compact interval and we assume 
\[
p_0^{-1}(I) =\bigset{\x\in M}{p_0(\x)\in I} \subset M_0, 
\]
and $p_0^{-1}(I)$ is compact. We now consider the scattering theory for the pair $(H, H_0)$ on the 
energy interval $I$, i.e., we study the absolutely continuous spectrum of
$H$ on $I$. We denote the spectral projection of an operator $A$ on $J\subset \re$ by 
$E_J(A)$. Then the wave operators
\[
W^I_\pm =\slim_{t\to\pm\infty} e^{itH} e^{-itH_0} E_I(H_0)
\]
exist and they are complete: $\Ran W^I_\pm =E_I(H)\mathcal{H}_{ac}(H)$. Moreover, 
the point spectrum $\s(H)\cap I$ is finite including the multiplicities (see Section~2). 

We write the energy surface of $H_0$ with an energy $\l\in I$ by 
\[
\Sigma_\l=p_0^{-1}(\{\l\}) =\bigset{\x\in M}{p_0(\x)=\l}.
\]
$\Sigma_\l$ is a regular submanifold in $M$, and we let $m_\l$ be the smooth density on $\Sigma_\l$ 
characterized as follows: $m_\l=i^* \tilde m_\l$, where $\tilde m_\l\in\bigwedge^{d-1}(M)$ such that 
$\tilde m_\l\wedge dp_0= m$, and $i\,:\, \Sigma_\l\hookrightarrow M$ is the embedding. 
(Note $m_\l$ is uniquely determined whereas $\tilde m_\l$ is not.) 
The scattering operator is defined by 
$S^I= (W^I_+)^* W^I_-$, $\mathcal{H}\to \mathcal{H}$, 
and it commutes with $H_0$. Hence $S^I$ is decomposed to a family of operators $\{S(\l)\}_{\l\in I}$, 
where $S(\l)$ is a unitary operator on $L^2(\Sigma_\l,m_\l)$ for a.e. $\l\in I$. $S(\l)$ 
is called the scattering matrix (see Section~5 for the detail). 
Our main result is the following: 

\begin{thm}\label{thm-main}
Under the above assumptions, $S(\l)$ is a pseudodifferential operator with its symbol 
in $S^{0}_{1,0}(\S_\l)$ for each $\l\in I\setminus \s_p(H)$. Moreover, 
\[
\mathrm{Sym}(S(\l))= e^{-i\g(x,\x)} +R(x,\x),
\]
where $\mathrm{Sym}(A)$ denotes the symbol of $A$, 
\begin{equation}\label{eq-Born}
\g(x,\x)= \int_{-\infty}^\infty V(x+tv(\x),\x) dt
\quad \text{for } \x\in\Sigma_\l, x\in T^*_\x\Sigma_\l, 
\end{equation}
and $R\in S^{-\m}_{1,0}(\Sigma_\l)$. 
\end{thm}

We note, in the right hand side of \eqref{eq-Born}, we identify $T^*_\x\Sigma_\l$ 
with a subspace of $T^*_\x M$ using the Riemannian metric. We also note 
$\g\in S^{-\m+1}_{1,0}(\Sigma_\l)$, and by the Taylor expansion, we have 
\[
e^{-i\g(x,\x)} =1-i\g(x,\x) +r(x,\x), \quad r\in S^{-2(\m-1)}_{1,0}(\Sigma_\l). 
\]
The first two terms in the right hand side corresponds to the classical Born approximation for the 
scattering matrix. 

\bigskip
The scattering matrix is one of the central objects in the scattering theory, and a large amount 
of effort has been devoted to the investigation, mostly for Schr\"odinger operators. 
Chapter~8 of Yafaev's textbook  \cite{Y2} is 
an excellent reference on this subject. However, the microlocal properties of the scattering matrix 
seem to have attracted not much attention. One of the pioneering works is a series of 
papers by Isozaki and Kitada \cite{IK1,IK2,IK3,IK4}, and they proved the off-diagonal smoothness 
of the scattering matrix using the so-called microlocal resolvent estimates. 
Yafaev used microlocal methods to study high energy asymptotics of the scattering matrix
\cite{Y1}. In these works, they have not given explicit representation of the symbol as in 
Theorem~\ref{thm-main}. For the scattering theory on scattering manifolds, Melrose and 
Zworski \cite{MZ} showed that the scattering matrices are Fourier integral operators 
(see also Ito and Nakamura \cite{IN} for a generalization). 

Recently, Bulger and Pushnitski have employed a sort of hybrid of the microlocal and 
the functional analytic methods to obtain spectral asymptotics of the scattering matrix 
(\cite{BP1,BP2}). In this paper we obtain analogous result for fixed energies using 
the standard pseudodifferential operator calculus on manifolds. We also mention closely related 
result on the spectral asymptotics by Birman and Yafaev \cite{BY}, of which our result 
may be considered as a refinement and a generalization, if we combine our result with the 
Weyl formula. 

One of the motivations of this work comes from applications to the scattering theory for 
discrete Schr\"odinger operators (see, e.g., Boutet de Monvel, and Sahbani \cite{BS}, 
Isozaki and Korotyaev \cite{IKo} and references therein). 
We can apply microlocal methods to the scattering theory of discrete Schr\"odinger 
operators by considering it as a problem on the Fourier space $\mathbb{T}^d$. 
In particular, we can show that the scattering matrix is a pseudodifferential 
operator on the energy surface embedded in the torus, provided the energy is non critical. 

\bigskip
We prepare estimates on the boundary value of resolvents, usually called the limiting 
absorption principle, using the Mourre theory in Section~2. In Section~3, we construct 
Isozaki-Kitada modifiers for our model. In Section~4 we prove microlocal resolvent 
estimates, and combining them we prove Theorem~\ref{thm-main} in Section~5. We generally follow 
the theory of Isozaki and Kitada \cite{IK1, IK2, IK3,IK4}, but with a somewhat different point of view. 
We discuss applications to operators on $\re^d$ in Section~6, and then applications to 
discrete Schr\"odinger operators in Section~7. 

In this paper, we employ slightly nonstandard notations on pseudo\-differential operator calculus.
Though $M$ is our configuration space, it is usually the Fourier variable space in applications. 
Thus $x\in T^*_\x M$, $\x\in M$, would be the space variable in the original model. 
In order to adjust to the standard notation in such applications, we express the cotangent bundle as
\[
T^*M =\bigset{(x,\x)}{\x\in M, x\in T^*_\x M}.
\]
Also, for a symbol $a(x,\x)$ on $T^*M$, we quantize it by 
\[
\Op(a)\f(\x) =a(-D_\x,\x)f(\x) 
=(2\pi)^{-d} \iint e^{-i(\x-\y)\cdot x} a(x,\y) \f(\y)d\y dx
\]
for $\f\in C_0^\infty(M)$ in a local coordinate system. We denote the composition of symbols $a$, $b$ 
by $a\#b$, i.e., $\Op(a\# b) =\Op(a)\Op(b)$. We denote the standard H\"ormander symbol class 
on $M$ by $S^m_{\rho,\delta}(M)$, $m\in\re$, $0\leq \d<\rho\leq 1$. Namely, $a\in S^{m}_{\rho,\d}(M)$ 
if for any $\a,\b\in\ze_+^d$ there is $C_{\a\b}>0$ such that 
\[
\bigabs{\pa_x^\a\pa_\x^\b a(x,\x)}\leq C_{\a\b} \jap{x}^{m-\rho|\a|+\d|\b|}, 
\quad \text{for } \x\in M, x\in T^*_\x M,
\]
in a local coordinate. In the following, we use only the case $\rho=1$, $\d=0$. 

The Fourier transform is also defined with a different signature in the exponent, i.e., 
\[
\mathcal{F}^*f(x) =(2\pi)^{-d/2}\int e^{i\x\cdot x} f(\x) d\x, \quad f\in C_0^\infty(\re^d), 
\]
is the Fourier transform from the $\re^d_\x$-space to the $\re^d_x$-space, 
and the definition of the wave front set is also changed, namely, the directions of 
singularities are reversed. 

We denote the Riemannian metric by $(g_{ij}(\x))$, and length of a vector in $T^*M$, inner products, etc., 
are defined using this metric. The densities $m$ and $m_\l$ are not necessarily the Riemannian 
densities. For a pair of non-zero vectors $v, w\in\re^d$, we denote 
\[
\cos(v,w) = \frac{v\cdot w}{|v|\,|w|} \in [-1,1]. 
\]

\noindent{\bf Acknowledgement} The author thanks Alexander Pushnitski for valuable discussion. 
The paper is very much motivated by discussions with him during his visit to Tokyo, 2013. 


\section{Limiting absorption principle}

Here we prepare basic estimates on the boundary value of resolvents using the Mourre theory \cite{M}. 
These results are essentially not new (see, e.g., Amrein, Boutet de Monvel, Georgescu \cite{ABG}, 
Section~7.6), and we briefly explain the proof partly for the completeness, but also because 
the formulation is slightly different. 

We choose $I'\Supset I$ so that $p_0^{-1}(I')\subset M_0$. We also choose $\chi_1\in C_0^\infty(M)$ 
such that $\supp \chi_1\subset p_0^{-1}(I')$ and $\chi_1=1$ on $p_0^{-1}(I)$. 
Then we define a vector field $A_0$ by 
\[
A_0=\sum_{j,k=1}^d \chi_1(\x) g^{jk}(\x)\frac{\pa p_0}{\pa \x_k}(\x)\frac{\pa}{\pa \x_j}
\]
in a local coordinate. Then we set
\[
A=\frac12 \bigpare{iA_0-iA_0^*} \quad \text{on } C_0^\infty(M).
\]
$A$ is essentially self-adjoint, and we denote the unique self-adjoint extension by the same symbol. 
Then it is easy to see 
\[
[H_0,iA] =\chi_1(\x)|v(\x)|^2
\]
and that $e^{i\s A}H_0e^{-i\s A}$ and $e^{i\s A}Ve^{-i\s A}$ are $H_0$-bounded operator valued 
$C^\infty$ functions in $\s\in\re$. It is also easy to see that $[V,iA]$ is a compact operator under 
our assumptions. Thus we can apply the Mourre theory on $I$. Moreover, since $A$ is relatively 
bounded with respect to $|D_\x|$, we can conclude the following standard result in two-body scattering theory. 

\begin{thm}\label{thm-LAP}
(1) $I\cap \s_p(H)$ is discrete, and it is finite with their multiplicities. \newline
(2) Let $s>1/2$. Then for any $\l\in I\setminus \s_p(H)$, 
\[
(H-\l\mp i0)^{-1} =\lim_{\e\dto 0} (H-\l\mp i\e)^{-1}
\]
exist as operators from $H^s(M)$ to $H^{-s}(M)$, and they are H\"older continuous in $\l$. 
In particular, the spectrum of $H$ is absolutely continuous on $I\setminus \s_p(H)$. \newline
(3) Let $k\in\na$ and let $s>k+1/2$. Then for $\l\in I\setminus \s_p(H)$, 
\[
(H-\l\mp i0)^{-k-1} =\lim_{\e\dto 0} (H-\l\mp i\e)^{-k-1}
\]
are bounded from $H^s(M)$ to $H^{-s}(M)$, and they are H\"older continuous in $\l$. 
In particular, $(H-\l\mp i0)^{-1}$ are $C^k$-class functions in $\l\in I\setminus\s_p(H)$ 
as operators from $H^s(M)$ to $H^{-s}(M)$. 
\end{thm}

For the abstract Mourre theory, we refer Mourre \cite{M}, Jensen, Mourre, Perry \cite{JMP},
Amrein, Boutet de Monvel, Georgescu \cite{ABG} and G\'erard \cite{Ge}. 


\section{Isozaki-Kitada modifiers}

Here we construct Isozaki-Kitada type modifiers for our model. For the short range perturbation, 
we can construct the modifiers as pseudodifferential operators (see, Isozaki, Kitada \cite{IK4}). 
We employ a slightly different construction from \cite{IK4}, which is already used, for example, 
in \cite{N1}. 

Let $I\Subset\re$ as before, and let 
\[
\O_\pm^\e =\bigset{(x,\x)\in T^*M}{\pm \cos(x,v(\x))>-1+\e, |x|>1, p_0(\x)\in I}
\]
with $\e>0$. We first construct symbols $a^\pm$ on $T^*M$ such that, roughly speaking, 
\[
H\Op(a^\pm) -\Op(a^\pm) H_0 \sim 0 \quad \text{in }\O_\pm^\e, 
\]
and $a^\pm(x,\x)\to 1$ as $|x|\to\infty$ in $\O_\pm^\e$. We construct $a^\pm(x,\x)$ of the form:
\[
a^\pm(x,\x)\sim e^{i\g_\pm(x,\x)}\bigpare{1+a_1^\pm(x,\x) +a_2^\pm(x,\x) +\cdots},
\]
where 
\[
\g_\pm(x,\x) =\int_0^{\pm \infty} V(x+t v(\x),\x)dt
\]
and $a_j^\pm\in S_{1,0}^{-\m+1-j}(M)$. 

We note, if $(x,\x)\in \O_\pm^\e$, then $\pm\cos(x+tv(\x),\x)>-1+\e$ for $\pm t\geq 0$, and 
\begin{align*}
|x+tv(\x)|^2 &= |x|^2 +t^2|v(\x)|^2 +2tx\cdot v(\x) \\
&\geq |x|^2 +t^2|v(\x)|^2 -2|t|(1-\e)|x|\cdot |v(\x)| \\
&\geq \e(|x|^2+|tv(\x)|^2) \geq \frac{\e}{2} (|x|+|tv(\x)|)^2.
\end{align*}
This implies, in particular, $\g_\pm\in S^{-\m+1}_{1,0}$ on $\O_\pm^\e$. 
We also note $\g_\pm$ satisfies the transport equation:
\[
v(\x)\cdot \pa_x \g_\pm(x,\x) +V(x,\x)=0.
\]
We set $a_0^\pm(x,\x)=1$, and we let
\begin{align*}
r^\pm_j(x,\x) &= e^{-i\g_\pm}\bigpare{p_0\# (e^{i\g_\pm}a_j^\pm)}
-\bigpare{p_0a_j^\pm -iv(\x)\cdot \pa_x a_j^\pm +V a_j^\pm}, \\
V_j^\pm(x,\x) &= e^{-i\g_\pm}\bigpare{V\# (e^{i\g_\pm}a_j^\pm)}-V a_j^\pm,
\end{align*}
for $j=0,1,2,\dots$. 
We compute the symbol of $H\Op(a^\pm)-\Op(a^\pm)H_0$  formally:  
\begin{align*}
&(p_0+V)\#\bigpare{e^{i\g_\pm}(1+a_1^\pm+\cdots)} -e^{i\g_\pm}\bigpare{1+a_1^\pm+\cdots}p_0 \\
&\quad = e^{i\g_\pm}\Bigbra{\sum_{j=1}^\infty (-iv(\x)\cdot\pa_x a_j^\pm(x,\x))
+\sum_{j=0}^\infty r_j^\pm(x,\x) +\sum_{j=0}^\infty V_j^\pm(x,\x)}. 
\end{align*}
We solve the following equations iteratively: 
\[
v(\x)\cdot \pa_x a_j^\pm(x,\x) + i r_{j-1}^\pm(x,\x) +i V_j^\pm(x,\x)=0, \quad j=1,2,\dots. 
\]
We note $r_0^\pm, V_0^\pm \in S^{-\m-1}_{1,0}$ on $\O_\pm^\e$. We choose solutions as follows: 
\[
a_j^\pm(x,\x) =i\int_0^{\pm\infty} \bigpare{r_{j-1}^\pm(x+tv(\x),\x) +V_{j-1}^\pm(x+tv(\x),\x)} dt
\]
for $j=1,2,\dots$, so that $a_j^\pm\in S^{-\m+1-j}_{1,0}$ and hence 
$r_j^\pm, V_j^\pm \in S^{-\m-1-j}_{1,0}$ on $\O_\pm^\e$, iteratively. 
Then we define $a^\pm$ as an asymptotic sum 
\begin{equation}
a^\pm(x,\x) \sim e^{i\g_\pm(x,\x)} \bigpare{1+a_1^\pm(x,\x)+a_2^\pm(x,\x)+\cdots}, 
\end{equation}
which is in $S^0_{1,0}$, and $(p_0+V)\# a^\pm -a^\pm\# p_0\in S^{-\infty}_{1,0}$ 
on $\O_\pm^\e$. 

Now we introduce microlocal cut-off to define operators $J_\pm$. 
Let $\y\in C^\infty([-1,1])$ be such that 
\[
\y(s)=\begin{cases} 1 \quad &\text{if }s>-1+\e, \\ 0 &\text{if }s<-1+\e/2. \end{cases}
\]
We fix $I_0\Subset I$, and choose $\chi\in C_0^\infty(\re)$ such that 
$\chi=1$ on $I_0$ and $\supp[\chi]\subset I$. Then we set 
\begin{equation} \label{eq-total-symbol}
\tilde a^\pm(x,\x) =\chi(p_0(\x))\y(\pm \cos(x,v(\x))) a^\pm(x,\x), 
\quad\x\in M, x\in T^*_\x M. 
\end{equation}
It is easy to see $\tilde a^\pm\in S^0_{1,0}(M)$, globally, and we define 
\[
J_\pm =\Op(\tilde a^\pm), 
\]
for given $\e>0$ and $I_0\Subset I$. It is straightforward to verify 
\[
G_\pm := H J_\pm -J_\pm H_0 =\Op(g^\pm)
\quad\text{with } g_\pm \in S^{-1}_{1,0}(M), 
\]
and 
\begin{align}
\esssupp[g_\pm]\subset&\bigset{(x,\x)}{p_0(\x)\in I, -1+\e/2\leq \pm \cos(x,v(\x)) \leq -1+\e}
\nonumber \\
&\cup \bigset{(x,\x)}{p_0(\x)\in I\setminus I_0}, \label{eq-G-remainder}
\end{align}
where $\esssupp[a]$ denotes the essential support of a pseudodifferential operator or its symbol. 


\section{Microlocal resolvent estimates}

Here we discuss a generalization of the microlocal resolvent estimates due to Isozaki and 
Kitada \cite{IK1,IK3}. We discuss only the two-sided microlocal resolvent estimates, which are 
used later. Our formulation is closer to the H\"ormander type microlocal analysis 
than those in the papers by Isozaki and Kitada, and they are actually more precise.
About closely related phase space localization estimates, 
we also refer a work by Mourre \cite{M2} (see also G\'erard \cite{Ge} for an alternative proof). 

We fix $I_0\Subset I\setminus \s_p(H)$, and we consider the microlocal properties of 
$(H-\l\mp i0)^{-1}$ : $H^s(M)\to H^{-s}(M)$, $s>1/2$, for $\l\in I_0$. 
Let $K^\pm(\l)$ be the distribution kernels of $(H-\l\mp i0)^{-1}$. 
For a distribution $T$ on $M$, we denote the wave front set of $T$ by $\WF(T)$. 
When we discuss wave front sets of distributions on $M\times M$, we identify 
$T^*(M\times M)\cong T^*M \times T^* M$, 
and we denote a point in $T^*(M\times M)$ as 
$(x,\x,y,\y)\in T^*M\times T^*M$, where $\x,\y\in M$, $x\in T^*_\x M$, $y\in T^*_\y M$. 

Our microlocal resolvent estimates are formulated as follows. We denote
\begin{align*}
&\S^0= \bigset{(x,\x,-x,\x)}{\x\in M, x\in T^*_\x M}, \\
& \S^1_\pm(\l)  = \bigset{(x+tv(\x),\x,-x,\x)}{\x\in M,  x\in T^*_\x M, p_0(\x)=\l, \pm t\geq 0}, \\
&\S_\pm^2(\l) = \bigset{(tv(\x),\x)}{\x\in M, p_0(\x)=\l,\pm t\geq 0}\times T^*M, \\
&\S_\pm^3(\l) = T^*M\times \bigset{(tv(\x),\x)}{\x\in M,p_0(\x)=\l,\pm t\geq 0}. 
\end{align*}

\begin{thm} \label{thm-MLRE}
For $\l\in I\setminus \s_p(H)$, 
\[
\WF(K^\pm(\l))\subset \S^0 \cup \S_\pm^1(\l)\cup \S_\pm^2(\l)\cup \S_\pm^3(\l).
\] 
\end{thm}

\begin{rem}
$\S^0$ represents the identity map on $T^*M$, and it comes from pseudodifferential operator 
type properties. As we see in Corollary~4.4, $\S^1_\pm(\l)$ comes from the singularities of 
the free resolvents. $\S^2_\pm(\l)$ and $\S^3_\pm(\l)$ are generated by combinations of 
smooth off-diagonal propagations and the singularities of the free resolvents. 
\end{rem}

The microlocal resolvent estimates of Isozaki-Kitada follow easily from this: 

\begin{cor}
Let $-1<\m_-<\m_+<1$, and suppose $b_\pm(x,\x)\in S^0_{1,0}(M)$ satisfy 
\begin{align*}
&\esssupp[b_+] \subset \bigset{(x,\x)}{\cos(x,v(\x))\geq \m_+, p_0(\x)\in I_0}, \\
&\esssupp[b_-] \subset \bigset{(x,\x)}{\cos(x,v(\x))\leq \m_-, p_0(\x)\in I_0}.
\end{align*}
Let $P_\pm =\Op(b_\pm)$. Then $P_\mp^* (H-\l\mp i0)^{-1} P_\pm$, $\l\in I_0$, 
are smoothing operators. 
\end{cor}

We first consider the free resolvents, i.e.,  $(H_0-\l\mp i0)^{-1}$. 

\begin{lem}\label{lem-free-resolvent}
For $\l\in I$, 
\[
\WF\bigbrac{(p_0(\x)-\l\mp i0)^{-1}} = \bigset{(t v(\x),\x)}{\x\in M, p_0(\x)=\l, \pm t>0}. 
\]
\end{lem}

\begin{proof}
This is discussed, for example, in H\"ormander \cite{Ho} Vol.1, Example~8.2.6. 
We give a proof for the sake of completeness. 

It is obvious that the wave front is contained in $\bigset{(x,\x)}{p_0(\x)=\l}$, and it suffices 
to consider the case $p_0(\x)=\l$. 
By a partition of unity and a change of coordinates, 
we may assume $\l=0$, $\x=0$ and $p_0(\x) =\x_1$ in a neighborhood of 0. 
In this case, $v(\x)=(1,0,\dots, 0)$. 

We note 
\[
\mathcal{F}^*_1[(\x\mp i0)^{-1}] =\pm \sqrt{2\pi}i F(\pm x), \quad x\in\re, 
\]
where $\mathcal{F}_1^*$ is the one-dimensional Fourier transform, and $F(x)$ is the 
characteristic function of $(0,\infty)$.  Let $\f_1\in C_0^\infty(\re)$, $\f_2\in C_0^\infty(\re^{d-1})$. 
Then by writing $\x=(\x_1,\x')$, we have 
\[
\mathcal{F}^*\bigbrac{\f_1(\x_1)\f_2(\x')(\x_1\mp  i0)^{-1}}(x_1,x') 
=\pm i (\check \f_1* F)(\pm x_1)\check\f_2(x'), \quad (x_1,x')\in\re^d.
\]
This implies $\WF\bigbrac{(\x_1\mp  i0)^{-1}}=\bigset{(t,0, 0, \x')}{\pm t>0, \x'\in\re^{d-1}}$. 
\end{proof}

Let $K_0^\pm(\l)$ be the distribution kernel of $(H_0-\l\mp i0)^{-1}$. 
Then the above lemma implies the following characterization of the wave front set of $K_0^\pm(\l)$. 

\begin{cor}
For $\l\in I$, $\WF\bigbrac{K_0^\pm(\l)} \subset\S^0\cup \S^1_\pm(\l)$. 
\end{cor}

\begin{proof}
It is easy to see 
\[
K_0^\pm(\x,\y)= (p_0(\x)-\l\mp i0)^{-1}\d(\x-\y), \quad \x,\y\in M,
\]
and the claim follows from a general theorem, e.g., \cite{Ho} Vol.1, Proposition~8.2.10 and 
Lemma~\ref{lem-free-resolvent}. 
One can also compute the wave front set directly, and show the equality actually holds. 
\end{proof}

\begin{proof}[Proof of Theorem~\ref{thm-MLRE}]
We fix $\l\in I_0\Subset I\setminus \s(H)$, and we consider the ``$+$'' case only. The ``$-$'' case is proved similarly. 
We suppose 
\[
(x_1,\x_1,-x_2,\x_2)\notin \S^0 \cup \S_+^1(\l)\cup \S_+^2(\l)\cup \S_+^3(\l),
\] 
and we show $(x_1,\x_1,-x_2,\x_2)\notin \WF(K^+(\l))$. We consider several cases separately. 

\noindent{\bf Case 1:} At first we consider the case $p_0(\x_1)=p_0(\x_2)=\l$. 
Since $x_1\notin \{tv(\x_1)\,|\,t>0\}$ and $x_2\notin\{tv(\x_2)\,|\,t<0\}$, we can choose $0<\e\ll 1$ so small that 
\[
\cos(x_1,v(\x_1))<1-2\e, \quad \cos(x_2,v(\x_2))>-1+2\e.
\]
Then we can choose $\i_1,\i_2\in S^0_{1,0}(M)$ so that they are homogeneous of order 0 in $x$
when $|x|>1$,  $\i_j(x_j,\x_j)>0$, $\i_j$ are supported in small conic neighborhoods of $(x_j,\x_j)$ 
and hence 
\begin{align*}
&\supp[\i_1] \subset \bigset{(x,\x)}{\cos(x,v(\x))<1-\e}, \\
&\supp[\i_2]\subset \bigset{(x,\x)}{\cos(x,v(\x))>-1+\e}.
\end{align*}
We then construct $J_\pm$ with this $\e>0$ and $I_0$ as in the last section. By the construction, 
we have 
\begin{align*}
&J_+(H_0-z)^{-1} -(H-z)^{-1}J_+ =(H-z)^{-1} G_+ (H_0-z)^{-1},  \\
&(H_0-z)^{-1} J_-^* -J_-^* (H-z)^{-1} =(H_0-z)^{-1}G_-^* (H-z)^{-1}
\end{align*}
for $z\in\co\setminus\re$. Combining them, we obtain 
\begin{align}
&J_-^* (H-\l- i0)^{-1} J_+ = J_-^* J_+ (H_0-\l-i0)^{-1} \nonumber \\
&\qquad - (H_0-\l-i0)^{-1} J_-^* G_+(H_0-\l-i0)^{-1} \nonumber \\
&\qquad - (H_0-\l-i0)^{-1} G_-^* (H-\l-i0)^{-1} G_+ (H_0-\l-i0)^{-1}. \label{eq-decomp}
\end{align}
By Corollary~4.4, we learn 
\begin{description}
\item[(C-1)]The wave front set of the distribution kernel of $J_-^* J_+ (H_0-\l-i0)^{-1}$ is a subset 
of $\S^0\cup \S_+^1$. 
\end{description}

Using \eqref{eq-G-remainder} and Corollary~4.4 again, we learn 
\begin{description}
\item[(C-2)] $\Op(\i_1) (H_0-\l-i0)^{-1} J_-^*$ maps $C^\infty$ functions to $C^\infty$ functions, 
\item[(C-3)] $G_+ (H_0-\l-i0)^{-1} \Op(\i_2)$ is smoothing,
\end{description}
provided $\i_1$ and $\i_2$ are supported in sufficiently small conic neighborhoods 
of $(x_1,\x_1)$ and $(x_2,\x_2)$, respectively. Thus we obtain 
\begin{description}
\item[(C-4)] $\Op(\i_1) (H_0-\l-i0)^{-1} J_-^* G_+(H_0-\l-i0)^{-1}\Op(\i_2)$ is smoothing.
\end{description}

Similarly, we have 
\begin{description}
\item[(C-5)] $\Op(\i_1) (H_0-\l-i0)^{-1} G_-^*$ and $G_+(H_0-\l-i0)^{-1} \Op(\i_2)$ are smoothing,
\item[(C-6)] $(H-\l-i0)^{-1}$ is bounded from $H^s(M)$ to $H^{-s}(M)$ with $s>1/2$ by Theorem~\ref{thm-LAP}.
\end{description}
Combining them, we learn 
\begin{description}
\item[(C-7)] $\Op(\i_1) (H_0-\l-i0)^{-1} G_-^* (H-\l-i0)^{-1} G_+ (H_0-\l-i0)^{-1}\Op(\i_2)$ is smoothing. 
\end{description}

From (C-4) and (C-7), we learn 
\begin{description}
\item[(C-8)] $\Op(\i_1) (J_-^* (H-\l- i0)^{-1} J_+-J_-^* J_+ (H_0-\l-i0)^{-1})\Op(\i_2)$ is a smoothing 
operator, and hence its distribution kernel has no wave front set. 
\end{description}
Noting  $(x_1,\x_1)\in\esssupp[\Op(\i_1)J_-^*]$ and $(x_2,\x_2)\in\esssupp[J_+\Op(\i_2)]$, 
we conclude from \eqref{eq-decomp}, (C-1) and (C-8) that $(x_1,\x_1,-x_2,\x_2)\notin\WF[K^+(\l)]$. 

\medskip
\noindent{\bf Case 2:} We now consider the case $p_0(\x_1)\neq \l$, $p_0(\x_2)\neq \l$. 
We use the following energy localization lemma. 

\begin{lem}\label{lem-energy-localization}
Suppose $\chi\in C_0^\infty(M)$, real-valued, and $\supp[\chi]\cap \S_\l=\emptyset$. 
Let $s>1/2$. 
Then there is $Q\in \OPS_{1,0}^0(M)$, with its symbol 
supported in an arbitrarily small neighborhood of $\supp[\chi]$,  such that 
\begin{enumerate}
\renewcommand\labelenumi{{\rm (\theenumi)} }
\item $\chi (H-\l\mp i0)^{-1}-\chi Q$ is bounded from $H^s(M)$ to $C^\infty(M)$;
\item $(H-\l\mp i0)^{-1}\chi -Q\chi$ is bounded from $\mathcal{E}'(M)$ to $H^{-s}(M)$;
\item $\chi (H-\l\mp i0)^{-1}\chi -\chi Q\chi$ is smoothing, i.e., bounded from $\mathcal{E}'(M)$ to
$C^\infty(M)$.  
\end{enumerate}
\end{lem}

\begin{proof}
Since $H-\l$ is elliptic on $\supp[\chi]$, one can construct a parametrix 
$Q\in \OPS_{1,0}^0(M)$ such that 
\[
\chi-(H-\l)Q\chi =R_1
\]
is smoothing. Moreover, we may assume $Q$ is supported in an arbitrarily small neighborhood of 
$\supp[\chi]$, and it is self-adjoint. Then we have 
\[
(H-\l\mp i0)^{-1}\chi -Q\chi =(H-\l\mp i0)^{-1}R_1,
\]
and 
\[
\chi (H-\l\mp i0)^{-1}-\chi Q =R_1^* (H-\l\mp i0)^{-1}. 
\]
The claims (1) and (2) follow from the above expressions and the limiting absorption principle, 
Theorem~\ref{thm-LAP}, respectively. Combining them, we learn 
\[
R_2=\chi (H-\l\mp i0)^{-1}\chi -\chi Q\chi
\]
is bounded from $H^s(M)$ to $H^k(M)$, and also bounded from $H^{-k}(M)$ to $H^{-s}(M)$, 
where $s>1/2$ and $k$ is an arbitrary integer. By interpolation, we conclude that $R_2$ is bounded 
from $H^{-(k-s)/2}(M)$ to $H^{(k-s)/2}(M)$, and this implies $R_2$ is smoothing.
\end{proof}

 We choose $\chi\in C_0^\infty(M)$ so that $\chi(\x_1)=\chi(\x_2)=1$ and 
 $\supp[\chi]\cap \S_\l=\emptyset$. Then by Lemma~\ref{lem-energy-localization} (3), we learn that 
 $\chi (H-\l\mp i0)^{-1}\chi -\chi Q\chi$ is smoothing, and hence the wave front set of the 
 distribution kernel of  $\chi (H-\l\mp i0)^{-1}\chi$ is the same as that of $\chi Q\chi$. 
 Since $Q$ is a pseudodifferential operator, the wave front set of the kernel is contained in $\S^0$. 
 Noting $\chi(\x_1)\chi(\x_2)=1$ and $(x_1,\x_1,-x_2,\x_2)\notin\S^0$, we conclude 
 $(x_1,\x_1,-x_2,\x_2)\notin \WF(K^+(\l))$. 
 
 \medskip
\noindent{\bf Case 3:} Suppose $p_0(\x_1)= \l$ and $p_0(\x_2)\neq \l$. We combine the above 
arguments. We choose $\e>0$ so that $\cos(x_1,v(\x_1))<1-2\e$, and we construct $J_-$ as in 
Case~1. We also choose $\chi\in C_0^\infty(M)$ so that $\chi(\x_2)=1$ and 
$\supp[\chi]\cap\S_\l=\emptyset$ as in Case~2. 
We then choose $\i_1\in S^0_{1,0}(M)$ so that $\i_1(x_1,\x_1)=1$, 
$\supp[\i_1]\subset \{\cos(x,v(\x_1))<1-\e\}$, and 
\[
\supp[\i_1]\cap \bigset{(x,\x)}{\x\in\supp[\chi],x\in T^*_\x M}=\emptyset.
\]
By Lemma~\ref{lem-energy-localization}, we have 
\begin{align*}
\Op(\i_1) J_-^* (H-\l-i0)^{-1}\chi &= \Op(\i_1) (H_0-\l-i0)^{-1} J_-^*\chi \\
& - \Op(\i_1) (H_0-\l-i0)^{-1} G_-^* Q\chi\\
& -\Op(\i_1) (H_0-\l-i0)^{-1} G_-^* (H-\l-i0)^{-1}R,
\end{align*}
where $R$ is a smoothing operator. As in Case~1 and Case~2, we can show that each term in the 
right hand side is smoothing. This implies  $(x_1,\x_1,-x_2,\x_2)\notin \WF(K^+(\l))$. 

\medskip
\noindent{\bf Case 4:} The case $p_0(\x_1)\neq \l$ and $p_0(\x_2)= \l$ is handled similarly to Case~3, 
and we omit the detail. 
\end{proof}


\section{Scattering matrices}

In this section we apply the results of previous sections to the scattering theory. 

\begin{prop}
Let $I_0\Subset I\setminus\s_p(H)$ be as in the previous sections. Then the wave operators 
\[
W_\pm^I =\slim_{t\to\pm\infty} e^{itH} e^{-itH_0} E_I(H_0)
\]
exists and they are complete, i.e., $\Ran W_\pm^I =E_I(H)\mathcal{H}_{ac}(H)$. 
\end{prop}

\begin{proof}
This is a standard argument, and we recall it briefly for completeness. 
Let $\f\in C_0^\infty(p_0^{-1}(I))$. Then by the non-stationary phase method, we learn that for any 
$N\in\na$, 
\[
\Bignorm{\chi\bigpare{-D_\x/\e t} e^{-itH_0} \f} \leq C_N \jap{t}^{_N}, \quad t\in\re, 
\]
where $\chi\in C_0^\infty(\re^d)$ is a smooth cut-off function 
such that $\chi(x)=1$ on $\{|x|\leq 1/2\}$ and $\supp[\chi]\subset \{|x|\leq 1\}$, and 
\[
0<\e<\inf\bigset{|v(\x)|}{\x\in\supp[\f]}.
\]
The existence of the wave operators follows easily by this and the Cook-Kuroda method. 
The completeness follows from the limiting absorption principle, Theorem~\ref{thm-LAP}, combined with, 
for example, the smooth perturbation theory (see, e.g., Reed-Simon \cite{RS} Section~VIII.7). 
\end{proof}

Then the scattering operator is defined by $S^I =(W_+^I)^*W_-^I$, and it is unitary on 
$E_I(H_0)\mathcal{H}= L^2(p_0^{-1}(I),m)$. Now $L^2(p_0^{-1}(I),m)$ is decomposed to 
\[
L^2(p_0^{-1}(I),m) \cong \int_I^\oplus L^2(\S_\l,m_\l)d\l,
\]
and the identification is given by the standard trace operator:
\[
T(\l)\f(\x) =\f(\x), \quad \text{for }\x\in\S_\l, \ \f\in H^s(M),
\]
where $s>1/2$. Since $S^I$ commutes with $H_0$, it is decomposed to a family of operators 
$\bigset{S(\l)\text{ on }L^2(\S_\l,m_\l))}{\l\in I}$ such that  
\[
S(\l) T(\l)\f = T(\l) S^I \f \quad \text{for } \f\in H^s(p_0^{-1}(I)).
\]

\begin{lem}\label{lem-energy-trace}
Let $T(\l)$ as above. Then
\[
T(\l)^* T(\l) =\d(H_0-\l) :=\frac1\pi \Im \bigbrac{(H_0-\l-i0)^{-1}}
\]
\end{lem}

\begin{proof}
By a partition of unity and a change of coordinates, we may assume 
$M=\re^d$, $\S_\l=\bigset{(0,\x')}{\x'\in\re^{d-1}}$, and consider the operators in a small neighborhood 
of $0\in\re^d$. We denote the velocity on $\S_\l$ by $dp_0(0,\x')=(v(\x'),0,\dots,0)$, 
the densities on $M$ and $\S_\l$ by 
$m=m(\x_1,\x')d\x_1d\x'$ and $m_\l =m_\l(\x')d\x'$, respectively. By the normalization of $m_\l$, we have 
\[
m(0,\x')=m_\l(\x')v(\x'), \quad \x'\in\re^{d-1}. 
\] 
We now compute the operator $T(\l)^*$: For $\f\in C_0^\infty(\re^d)$ and $\g\in C_0^\infty(\re^{d-1})$, 
\[
(\f,T(\l)^*\g)_M =(T(\l)\f,\g)_{\S_\l} = \int_{\S_\l} \f(0,\x')\overline{\g(\x')} m_\l(\x') d\x'.
\]
Hence, we have 
\[
T(\l)^*\g(\x)= \frac{m_\l(\x')}{m(0,\x')} \g(\x')\d(\x_1) = \frac{\g(\x')}{v(\x')}\d(\x_1).
\]
This implies 
\[
T(\l)^*T(\l)\f(\x) = {v(\x')}^{-1} \d(\x_1)\f(0,\x'), \quad \f\in C_0^\infty(\re^d).
\]
On the other hand, by the change of variables for distributions, we learn 
\[
\d(H_0-\l) =\d(p_0(\x)-\l) = v(\x')^{-1} \d(\x_1)
\]
and these completes the proof. 
\end{proof}

Let $J_\pm$ be the Isozaki-Kitada modifiers constructed in the previous section with $0<\e<1$. Then the following formula is 
well-known. 

\begin{prop}\label{prop-S-matrx-rep}
For $\l\in I\setminus \s_p(H)$, 
\begin{equation}\label{eq-s-matrix-rep}
S(\l) =-2\pi i T(\l) J_+^* G_- T(\l)^* 
+ 2\pi i T(\l) G_+^* (H-\l-i0)^{-1} G_- T(\l)^*.
\end{equation}
\end{prop}

For the proof, we refer Yafaev \cite{Y1}, and a corresponding formula in Isozaki-Kitada \cite{IK4} 
is essentially equivalent. The proof is functional analytic, and the computation can be carried out 
without any changes under our setting with the help of Lemma~\ref{lem-energy-trace}.  
In order to compute the right hand side terms, we use the 
following lemma. 

\begin{lem}\label{lem-sym}
Let $M$ be a manifold with a smooth density $m$, and let $\L$ be a smooth submanifold of 
codimension one with a smooth density $\tilde m$. 
Let $T$ be the trace operator to $\L$ : $H^s(M,m)\to L^2(\L,\tilde m)$, where $s>1/2$.
We denote the normal vector at $\x\in \L$ by $n(\x)\in T^*_\x M$ normalized so that 
$m=\hat m\wedge n$, where $\hat m\in\bigwedge^{d-1}(M)$, $i^* \hat m=\tilde m$, 
and $i\,:\,\L\hookrightarrow M$ is the embedding. 
\begin{enumerate}
\renewcommand\labelenumi{{\rm (\theenumi)} }
\item Suppose $a\in S^{-\m}_{1,0}(M)$ with $\m>1$. Then $T\Op(a) T^*$ is a pseudo\-differential operator
with its symbol in $S^{-\m+1}_{1,0}(\L)$. 
\item Suppose $a\in S^k_{1,0}(M)$, $k\in\re$, and suppose 
\[
\esssupp[a]\cap \bigset{(\pm n(\x),\x)}{\x\in\L}=\emptyset. 
\]
Then $T\Op(a) T^*$ is a pseudodifferential operator with its symbol in $S^{k+1}_{1,0}(\L)$. 
\item Suppose either the condition of (1) or (2) is satisfied. Then the principal symbol of 
$T\Op(a)T^*$ is given by 
\begin{equation}\label{eq-rest-formula}
\tilde a(x,\x)= \frac{1}{2\pi}\int_{-\infty}^\infty a(x+tn(\x),\x) dt, \quad \x\in\L, x\in T^*_\x\L.
\end{equation}
\end{enumerate}
\end{lem}

\begin{rem}
We note that $x+tn(\x)$ in \eqref{eq-rest-formula} is not necessarily well-defined, since 
there is no canonical embedding of $T^*_\x\L$ into $T^*_\x M$. However, the kernel of 
the canonical projection: $i^*\,:\,T^*_\x M\to T^*_\x\L$ is spanned by the normal vector $n(\x)$, 
and the integral in  \eqref{eq-rest-formula} is invariant under the translation: 
$(x,\x)\mapsto (x+sn(\x),\x)$, $s\in\re$. Hence $\tilde a(x,\x)$ is well-defined as a function on 
$T^*\L$. If $M$ is equipped with a Riemannian metric, we can naturally identify $T_\x^*\L$
with the normal subspace $\{n(\x)\}^\perp\subset T^*_\x M$, and the definition is simpler. 
\end{rem}

\begin{proof}
By a partition of unity and a change of coordinates, we may assume 
$M=\re^d$, $\L=\bigset{(0,\x')}{\x'\in\re^{d-1}}$, and $a$ is supported in a small neighborhood 
of $0\in\re^d$. We denote the normal vector by $(n(\x'),0,\dots,0)$, the densities on $M$ and $\L$ by 
$m=m(\x_1,\x')d\x_1d\x'$ and $\tilde m =\tilde m(\x')d\x'$, respectively. 
Analogously to the proof of Lemma~\ref{lem-energy-trace}, we have 
\[
T^*\g(\x)= \frac{\tilde m(\x')}{m(0,\x')} \g(\x')\d(\x_1) = \frac{\g(\x')}{n(\x')}\d(\x_1).
\]
We now compute $T\Op(a) T^*$ in the local coordinate: 
\begin{align*}
&T\Op(a)T^*\f(\x') = (2\pi)^{-d} \iint e^{-i((0,\x')-\y)\cdot x} a(x_1,x',\y_1,\y')
\frac{\f(\y')}{n(\y')}\d(\y_1)d\y dx\\
&\quad= (2\pi)^{-(d-1)}\iint e^{-i(\x'-\y')\cdot x'}
\biggpare{\frac{1}{2\pi}\int_{-\infty}^\infty a(x_1,x',0,\y')\frac{dx_1}{n(\y')}} \f(\y') d\y' dx' \\
&\quad = (2\pi)^{-(d-1)}\iint e^{-i(\x'-\y')\cdot x'}
\biggpare{\frac{1}{2\pi}\int_{-\infty}^\infty a(tn(\y'),x',0,\y')dt} \f(\y') d\y' dx'.
\end{align*}
The last expression proves \eqref{eq-rest-formula} for a suitable symbol $a(x,\x)$. 
We can show, by direct computations, that $\tilde a\in S^{1-\m}_{1,0}(\L)$ if 
$a\in S^{-\m}_{1,0}(M)$, $\m>1$. 
If $a\in S^k_{1,0}(M)$ satisfies the condition in (2), then in the local coordinate, we have 
\[
\esssupp[a]\subset \bigset{(x_1,x',\x_1,\x')}{\x\in\L, |x'|>\e |x_1|}
\]
with some $\e>0$.
Then it is straightforward to verify $\tilde a\in S^{k+1}_{1,0}(\L)$, and we can justify the 
above argument. 
\end{proof}

\begin{proof}[Proof of Theorem~\ref{thm-main}]
We recall the symbols of $J_\pm$ are given by \eqref{eq-total-symbol}. 
We denote
\[
Y(x,\x) =\y(-\cos(x,v(\x)), \quad \x\in M, x\in T^*_\x M.
\]
Then, by straightforward computations, we learn that the principal symbol of $J_+^* G_-$ is given by 
\[
\chi(p_0(\x))^2 e^{-i\g_+}(-i)\{p_0,Y\} e^{i\g_-} 
=-i e^{-i\g(x,\x)} \chi(p_0(\x))^2 v(\x)\cdot\pa_x Y(x,\x), 
\]
where 
\[
\g(x,\x) := \g_+(x,\x) -\g_-(x,\x) =\int_{-\infty}^\infty V(x+tv(\x),\x)dt.
\]
It is easy to see that $\g(x,\x)$ is invariant under the translation: $(x,\x)\mapsto (x+sv(\x),\x)$
for any $s\in\re$. Now we note 
\[
\lim_{t\to+\infty} Y(x+t v(\x),\x) =0, \quad 
\lim_{t\to-\infty} Y(x+t v(\x),\x) =1
\]
for any $\x\in p_0^{-1}(I)$ and $x\in T^*_\x M$. We also note 
\[
\frac{d}{dt} Y(x+tv(\x),\x) =v(\x)\cdot \pa_x Y(x+t v(\x),\x).
\]
Combining these, we have 
\begin{align*}
\int_{-\infty}^\infty v(\x)\cdot\pa_x Y(x+tv(\x),\x) dt 
&=\lim_{T\to\infty} (Y(x+Tv(\x),\x)-Y(x-Tv(\x),\x)) \\
&= -1.
\end{align*}
Since $J_+^* G_-$ is essentially supported away from $\{(\pm v(\x),\x)\}$, we can apply 
Lemma~\ref{lem-sym} (2) to learn that $T(\l)  J_+^* G_-T(\l)^*$ is a pseudodifferential operator and 
its principal symbol is given by  
\[
\frac{-i}{2\pi} \int_{-\infty}^\infty e^{-i\g(x,\x)} v(\x)\cdot \pa_x Y(x+tv(\x),\x) dt 
=\frac{i}{2\pi} e^{-i\g(x,\x)}.
\]
Hence the principal symbol of $-2\pi i T(\l) J^*_+ G_- T(\l)^*$ is given by $e^{-i\g(x,\x)}$. 
The second term in the right hand side of \eqref{eq-s-matrix-rep} is a smoothing operator 
by the microlocal resolvent estimate, 
Corollary 4.2, and we conclude that the principal symbol of $S(\l)$ is given by $e^{-i\g(x,\x)}$
modulo the $S^{-1}_{1,0}(\S_\l)$ terms.  

In order to obtain a better remainder estimate, we use the following trick. 
If $V=0$, then $S(\l)=1$ for all $\l\in I$. This also corresponds to the case $\g(x,\x)=0$ and 
$a^\pm(x,\x) =1$. Now we note 
\[
b^\pm(x,\x) := \tilde a^\pm(x,\x)-\chi(p_0(\x))\y(\pm\cos(x,v(\x))) \in S^{-\m+1}_{1,0}(M), 
\]
and we apply the above argument to $b^\pm(x,\x)$ to conclude that the principal symbol 
of $S(\l)-I$ is given by $e^{-i\g(x,\x)}-1$, and moreover, the remainder is contained in 
$S^{-\m}_{1,0}(\S_\l)$. Thus we conclude that the symbol of $S(\l)$ is $e^{-i\g(x,\x)}$ 
modulo the $S^{-\m}_{1,0}(\S_\l)$ terms. 
\end{proof}


\section{Applications to operators on Euclidean spaces}

Here we discuss applications of our main theorem to operators on Euclidean spaces, 
in particular Schr\"odinger type operators. 
In this section we let $M=\re^d$ and $\mathcal{H}=L^2(\re^d)$
with the standard Lebesgue measure. 

\begin{example}
We set $p_0(\x)$ to be a real-valued elliptic polynomial of order $2m$ on $\re^d$, 
and we set $H_0=p_0(D_\x)$ on $L^2(\re^d)$. 
We suppose $V(x,\x)$ is a $2m$-th order polynomial in $\x$ with 
smooth coefficients $\{a_\a(x)\}$, i.e., 
\[
V(x,\x) = \sum_{|\a|\leq 2m} a_\a(x) \x^\a, \quad x,\x\in\re^d.
\]
We suppose $a_\a(x)$ are real-valued and there is $\m>1$ such that for any $\b\in\ze_+^d$, 
\begin{equation}\label{eq-decay-ass}
\bigabs{\pa_x^\b a_\a(x)} \leq C_\b \jap{x}^{-\m-|\b|}, \quad x\in\re^d, \ |\a|\leq 2m.
\end{equation}
We quantize $V$ by 
\[
V= \frac12 \sum_{|\a|\leq 2m} \bigpare{a_\a(x) D_x^\a + D_x^\a a_\a(x)},
\]
then $V$ is an infinitesimally $H_0$-bounded symmetric operator. Hence $H=H_0+V$ 
is a self-adjoint operator. Then we can apply Theorem~1.1 for $\hat H= \mathcal{F} H \mathcal{F}^*$, 
provided $\l\in\re$ is a non-critical value of $p_0(\x)$. Thus the scattering matrix is a pseudodifferential 
operator with the principal symbol: 
\[
s_0(\l;x,\x) =e^{-i\g(\l;x,\x)}, \quad 
\g(\l;x,\x) = \sum_{|\a|\leq 2m} \int_{-\infty}^\infty a_\a(x+tv(\x))\x^\a dt,
\]
where $p_0(\x)=\l$, and $x\in \{v(\x)\}^\perp\cong T^*_\x \S_\l$. 

A typical example is the Schr\"odinger operator, i.e., 
\[
H = -\frac12 \triangle +V(x), 
\]
where $V(x)=a_0(x)$ is supposed to satisfy the condition \eqref{eq-decay-ass}. 
In this case, $v(\x)=\x$ and $\S_\l= \bigset{\x\in\re^d}{\frac12|\x|^2 =\l}$.
Then we recover the X-ray transform type approximation (\cite{BP1,BP2}), i.e., the 
principal symbol of the scattering matrix is given by 
$s_0(\l;x,\x) =e^{-i\g(\l;x,\x)}$, where
\[
\g(\l;x,\x) = \int_{-\infty}^\infty V(x+t\x)dt= \frac{1}{|\x|} \int_{-\infty}^\infty V(x+t\hat\x)dt
\]
with $\x\in \S_\l$, $x\perp \x$, $\hat\x=\x/|\x|$. In particular, $\g(\l;x,\x)$ is homogeneous 
of degree $(-1/2)$ with respect to the energy $\l=|\x|^2/2$. 
\end{example}

\begin{example}
Another typical example is the so-called relativistic Schr\"odinger operator. 
Let $g_{ij}(x)$ be a smooth Riemannian metric on $\re^d$, $W(x)$ be a smooth real-valued 
function,  and $m\geq 0$. We suppose there is $\m>1$ such that for any $\a\in\ze_+^d$, 
\[
\bigabs{\pa_x^\a (g_{ij}(x)-\d_{ij})}\leq C_\a \jap{x}^{-\m-|\a|}, \quad x\in\re^d, 
\]
and
\[
\bigabs{\pa_x^\a W(x)}\leq C_\a \jap{x}^{-\m-|\a|}, \quad x\in\re^d. 
\]
Then we define 
\[
H_0=\sqrt{-\triangle +m^2}, 
\]
and 
\[
H=\biggpare{\sum_{i,j=1}^d D_{x_j} g_{jk}(x) D_{x_k}+m^2}^{1/2} +W(x)
\]
on $\mathcal{H}=L^2(\re^d)$. It is easy to see that $H_0$ and $H$ are self-adjoint with 
$D(H_0)=D(H)=H^1(\re^d)$. 
Then we can show $\hat H_0=\mathcal{F} H_0 \mathcal{F}^{-1}$ and 
$\hat H=\mathcal{F} H \mathcal{F}^{-1}$ satisfy the assumptions of Theorem~\ref{thm-main}, 
and the result holds away from the critical value $\l=0$. 

We note that if $m=0$, then the symbol of $H_0$ is $|\x|$ and it has a singularity are $\x=0$. 
However, we can easily isolate the singularity using energy localization (see, e.g., 
Lemma~\ref{lem-energy-localization}). If, in addition, $g_{ij}(x)=\d_{ij}$, then we have 
\[
H=|D_x|+W(x), \quad \text{and}\quad v(\x)=\hat\x =\frac{\x}{|\x|}.
\]
By Theorem~\ref{thm-main}, the principal symbol of the scattering matrix is given by 
\[
s_0(\l;x,\x) =e^{-i\g(x,\x)}, \quad 
\g(x,\x) = \int_{-\infty}^\infty W(x+t\hat\x)dt ,
\]
where $|\x|=\l$, $x\perp \x$. We note these symbols are actually independent of the energy $\l>0$. 
\end{example}


\section{Applications to discrete Schr\"odinger operators}

In this section we discuss applications of our result to operators on 
the lattice $\ze^d$. We consider the Fourier space, or the dual group, $\torus^d$ as our 
configuration space, where $\torus=\re/(2\pi \ze)$. 

Let $\hat H_0$ be a self-adjoint translation invariant (i.e., constant coefficients) finite difference 
operator on $\ell^2(\ze^d)$, and let $V(n)$ ($n\in\ze^d$) be a multiplication operator on $\ze^d$. 
We consider scattering theory for the pair 
\[
\hat H_0\quad\text{and} \quad \hat H=\hat H_0+V\quad\text{on $\ell^2(\ze^d)$}.
\]
We denote the discrete Fourier transform by 
\[
F\f(\x) =(2\pi)^{-d/2}\sum_{n\in\ze^d} e^{-in\cdot \x} \f(n), \quad \x\in\torus^d, 
\]
which is unitary from $\ell^2(\ze^d)$ to $L^2(\torus^d)$. We denote the symbol of 
$\hat H_0$ by $p_0(\x)$, i.e., 
\[
p_0(\x) =(2\pi)^{d/2} F(\hat H_0 \d_0), 
\quad \text{where }\d_0(n)= \prod_{j=1}^d \d_{n_j0}\in\ell^2(\ze^d).
\]
By the self-adjointness of $\hat H_0$, $p_0$ is a real-valued trigonometric polynomial. We write
\[
H_0=F \hat H_0 F^* \quad \text{on }\mathcal{H}=L^2(\torus^d),
\]
and it is the multiplication operator by $p_0(\x)$. 

Now we denote the directional difference operators by 
\[
\tilde \pa_j \f(n) = \f(n)-\f(n-e_j), \quad n\in\ze^d, j=1,\dots, d, 
\]
where $(e_1,\dots,e_d)\subset \ze^d$ is the standard basis of $\re^d$. 
On the potential, we suppose:

\begin{ass}
$V\in\ell^\infty(\ze^d)$, real-valued, and there is $\m>1$ such that for any $\a\in\ze_+^d$,
\[
\bigabs{\tilde\pa^\a V(n)} \leq C_\a \jap{n}^{-\m-|\a|}, \quad n\in\ze^d, 
\]
with some $C_\a>0$. 
\end{ass}

Under this assumption, we can show that $V$ is extended to a real-valued smooth function 
$\tilde V$ on $\re^d$ such that for any $\a\in\ze_+^d$ 
\[
\bigabs{\pa_x^\a \tilde V(x)} \leq C_\a \jap{x}^{-\m-|\a|}, \quad x\in\re^d, 
\]
with some $C_\a>0$ (see, e.g., \cite{N2}, Lemma~2.1). 
We denote the standard Fourier transform on $\re^d$ by $\mathcal{F}$, and we let 
$\tilde V(-D_\x)= \mathcal{F} \tilde V(\cdot) \mathcal{F}^*$
be a Fourier multiplier on $\re^d$. 

\begin{lem}\label{lem-multipliers}
We identify $\torus^d \cong [-\pi,\pi)^d$, and let $\chi\in C_0^\infty((-\pi,\pi)^d)$. Then there 
is a smoothing operator $K$ on $\torus^d$ such that 
\[
\chi F V F^* \f = \chi \tilde V(-D_\x) \f +K\f, \quad \f\in C_0^\infty((-\pi,\pi)^d).
\]
Namely, $FVF^*$ and $\tilde V(-D_\x)$ coincide on $(-\pi,\pi)^d$ modulo the smoothing 
operators. 
\end{lem}

\begin{proof}
We use an operator $\Pi$ : $L^1(\re^d)\to L^1(\torus^d)$ defined by 
\[
(\Pi u)(\x) =\sum_{n\in\ze^d} u(\x+2\pi n), \quad \x\in\torus^d\cong [-\pi,\pi)^d, u\in L^1(\re^d).
\]
Then we have 
\[
\Pi \tilde V(-D_\x)\f(\x) = (2\pi)^{-d} \sum_n \iint e^{-i(\x-\y+2\pi n)\cdot x} \tilde V(x) \f(\y)d\y dx.
\]
By the Poisson summation formula: 
${\displaystyle \sum_{n \in\ze^d}  e^{2\pi i n\cdot x} =\sum_{m\in\ze^d} \d(x-m)}$, we learn 
\begin{align*}
\Pi \tilde V(-D_\x)\f(\x) &= (2\pi)^{-d}\sum_m \iint e^{-i(\x-\y)\cdot x} \tilde V(x) \d(x-m)\f(\y)d\y dx\\
&=(2\pi)^{-d}\sum_m \iint e^{-i(\x-\y)\cdot m} V(m)\f(\y)d\y \\
& =F VF^* \f(\x). 
\end{align*}
On the other hand, we write
\begin{align*}
& \chi \tilde V(-D_\x)\f(\x) -\chi \Pi \tilde V(-D_\x)\f(\x)
= \sum_{m\neq 0} \chi(\x)(\tilde V(-D_\x)\f)(\x+2\pi m)\\
&\quad =(2\pi)^{d/2} \int_{[-\pi,\pi)^d} \sum_{m\neq 0} \chi(\x)(\mathcal{F}\tilde V)(\x-\y+2\pi m)\f(\y)d\y \\
&\quad = \int_{[-\pi,\pi)^d} K(\x,\y) \f(\y) d\y
\end{align*}
with a smooth kernel $K(\x,\y)\in C^\infty((-\pi,\pi)^d\times (-\pi,\pi)^d)$. Thus 
\[
\chi F VF^* \f = \chi \tilde V(-D_\x)\f -\int K(\x,\y) \f(\y) d\y, 
\]
and this completes the proof. 
\end{proof}

We then consider $\tilde V(-D_\x)$ in the sense of pseudodifferential operator on $\torus^d$. 
Then Lemma~\ref{lem-multipliers} implies $\tilde V(-D_\x)$ and $F^*VF$ coincides modulo the smoothing operators,  
and thus we may consider $p(x,\x)= p_0(\x)+\tilde V(x)$ as the symbol of $H$. 
Now we can apply our results, in particular Theorem~\ref{thm-main} to our model. 
We consider more specific examples in the rest of this section. 

\begin{example}[Square lattice]
We consider discrete Schr\"odinger operators with the difference Laplacian: 
\[
\hat H_0\f(n)= \frac12\sum_{|n-m|=1} (\f(n)-\f(m)) \quad\text{for } n\in\ze^d, \, \f\in \ell^2(\ze^d),
\]
and we set $\hat H=\hat H_0+V$, where $V$ satisfies Assumption~A. Then it is easy to show 
\[
p_0(\x)= \sum_{j=1}^d (1-\cos(\x_j)), \quad\x\in\torus^d, 
\]
and hence $\s(H_0)=[0,2d]$. The velocity is given by 
\[
v(\x) =(\sin(\x_1),\dots,\sin(\x_d))\in\re^d, \quad \x\in\torus^d.
\]
We note $v(\x)=0$ if and only if $\sin(\x_j)=0$ for $j=1,\dots, d$. 
These critical points corresponds to the critical values, or the threshold energy sets,
 $\mathcal{T}=\{0,2,\dots, 2d\}$. 

The energy surface $\S_\l$, $\l\in[0,2d]\setminus\mathcal{T}$, is a regular submanifold, and 
it is diffeomorphic to the sphere $\mathbb{S}^{d-1}$, not unlike in the Euclidean space case. 
For $\l\in [0,2d]\setminus (\mathcal{T}\cup\s_p(H))$, 
the scattering matrix $S(\l)$ is defined as a unitary operator on $L^2(\S_\l,m_\l)$, and it is a 
pseudodifferential operator. Moreover, the principal symbol is given by the formula \eqref{eq-Born}
of Theorem~\ref{thm-main}. 
\end{example}

\begin{example}[2D triangular lattice] 
Here we consider 2 dimensional triangular lattice. The configuration space is also $\ze^2$, but the 
free Hamiltonian is given by 
\[
\hat H_0\f(n)= \frac12 \sum_{|n-m|=1} (\f(n)-\f(m)) 
+\frac12\sum_{j=\pm 1}(\f(n)-\f(n_1+j,n_2+j)), 
\]
for $\f\in\ell^2(\ze^2)$. Then the symbol is given by 
\[
p_0(\x)= 3-\cos(\x_1)-\cos(\x_2) -\cos(\x_1+\x_2), \quad \x=(\x_1,\x_2)\in\torus^2.
\]
By direct computations, we learn $v(\x)=0$ if and only if either (1) $\x_1=0,\pi$ and $\x_2=0,\pi$; 
or (2) $\x_1=\x_2=\pm\frac23\pi$. Thus $p_0(\x)$ has six critical points (one minimum, two maxima and 
three saddle points), and the critical values are $\mathcal{T}=\{0,2,\frac92\}$.  
The spectrum is $\s(H_0)=[0,\frac92]$. 

For $\l\in (0,2)$, $\S_\l$ is diffeomorphic to the circle $\mathbb{S}^1$; 
for $\l\in (2,\frac92)$, $\S_\l$ has two connected components, and each is diffeomorphic to $\mathbb{S}^1$. 
The scattering matrix $S(\l)$ is a pseudodifferential operator on such a manifold $\S_\l$ 
if  $\l\in ((0,2)\cup(2,\frac92))\setminus \s_p(H)$. 
\end{example}



\begin{thebibliography}{99}

\bibitem{ABG} Amrein, W., Boutet de Monvel, A.,  Georgescu, V.: 
$C_0$-groups, commutator methods and spectral theory of $N$-body Hamiltonians.
Progress in Mathematics, 135. Birkh\"auser Verlag, Basel, 1996. 

\bibitem{BY} Birman, M. Sh., Yafaev, D. R.: 
The asymptotic behavior of the spectrum of the scattering matrix. 
J. Soviet Math. {\bf 25} (1984), 793--814. 

\bibitem{BS} A. Boutet de Monvel, J. Sahbani : 
On the spectral properties of discrete Schr\"odinger operators: (The multi-dimensional case). 
Rev. Math. Phys. {\bf 11} (1999), 1061--1078. 

\bibitem{BP1}  Bulger, D.,  Pushnitski, A.: 
The spectral density of the scattering matrix for high energies. 
Comm. Math. Phys. {\bf 316} (2012), no. 3, 693--704.

\bibitem{BP2}  Bulger, D., Pushnitski, A.: 
The spectral density of the scattering matrix of the magnetic Schr\"odinger operator for high energies. 
J. Spectr. Theory {\bf 3} (2013), no. 4, 517--534.

\bibitem{Ge}  G\'erard, C.: 
A proof of the abstract limiting absorption principle by energy estimates. 
J. Funct. Anal. {\bf 254} (2008), no. 11, 2707--2724.
\bibitem{Ho} H\"ormander, L.: 
The Analysis of Linear Partial Differential Operators.  I--IV, Springer-Verlag, New York, 1983--1985. 

\bibitem{IK1} Isozaki, H., Kitada, H.: Microlocal resolvent estimates for 2-body Schr\"odinger operators. 
J. Funct. Anal. {\bf 57} (1984), no. 3, 270--300.

\bibitem{IK2} Isozaki, H., Kitada, H.: Modified wave operators with time-independent modifiers. 
J. Fac. Sci. Univ. Tokyo Sect. IA Math. {\bf 32} (1985), no. 1, 77--104.

\bibitem{IK3} Isozaki, H., Kitada, H.: A remark on the microlocal resolvent estimates for two body 
Schr\"odinger operators. 
Publ. Res. Inst. Math. Sci. {\bf 21} (1985), no. 5, 889--910.

\bibitem{IK4} Isozaki, H., Kitada, H.: Scattering matrices for two-body Schr\"odinger operators. 
Sci. Papers College Arts Sci. Univ. Tokyo {\bf 35} (1986), no. 2, 81--107.

\bibitem{IKo} H. Isozaki, I. Korotyaev: Inverse Problems, Trace Formulae for Discrete Schr\"odinger Operators.
Ann. Henri Poincar\'e {\bf 13} (2012), 751--788. 

\bibitem{IN} Ito, K., Nakamura, S.: 
Microlocal properties of scattering matrices for Schr\"odinger equations on scattering manifolds. 
Analysis and PDE {\bf 6} (2013), No. 2, 257--286. 

\bibitem{JMP} Jensen, A., Mourre, E., Perry, P.:
Multiple commutator estimates and resolvent smoothness in quantum scattering theory. 
Ann. Inst. H. Poincar\'e Phys. Th\'eor. {\bf 41} (1984), no. 2, 207--225. 

\bibitem{MZ} Melrose, R., Zworski, M.: 
Scattering metrics and geodesic flow at infinity, 
Invent. Math. {\bf 124} (1996), 389--436.

\bibitem{M} Mourre, E.: 
Absence of singular continuous spectrum for certain selfadjoint operators.
Comm. Math. Phys. {\bf 78} (1980/81), no. 3, 391--408. 

\bibitem{M2} Mourre, E.: Operateurs conjugu\'es et propri\'et\'es de propagation. 
Comm. Math. Phys. {\bf 91} (1983), no. 2, 279--300.

\bibitem{N1} Nakamura, S.: 
Time-delay and Lavine's formula.
Commun. Math. Phys. {\bf 109} (1987), 397--415.

\bibitem{N2} Nakamura, S.:
Modified wave operators for discrete Schr\"odinger operators with long-range perturbations.
Preprint, 2014 March. ({\tt http://arxiv.org/abs/1403.2795})

\bibitem{RS} Reed, M., Simon, B.: The Methods of Modern Mathematical Physics, 
Volume III, Scattering Theory, Academic Press, 1979. 

\bibitem{Y1}  Yafaev, D. R.: High-energy and smoothness asymptotic expansion of the scattering amplitude. 
J. Funct. Anal. {\bf 202} (2003), no. 2, 526--570.

\bibitem{Y2}  Yafaev, D. R.: Mathematical scattering theory. Analytic theory. 
Mathematical Surveys and Monographs, 158. American Mathematical Society, Providence, RI, 2010. 

\end{thebibliography}
\end{document}